\documentclass[12pt]{article}
\usepackage{latexsym,palatino,amsthm,amssymb,amsmath,amsfonts,graphicx,color,tikz}
\usepackage[colorlinks=true,citecolor=black,linkcolor=black,urlcolor=gray]{hyperref}
\usepackage{here}
\usepackage{comment}
\usepackage{url}
\allowdisplaybreaks
\newtheorem{theorem}{Theorem}[]
\newtheorem{prediction}{Prediction}[]
\newtheorem{definition}[theorem]{Definition}
\newtheorem{example}[theorem]{Example}

\newtheorem{lemma}[theorem]{Lemma}

\newtheorem{computer}[theorem]{Computer Calculation}
\newtheorem{conjecture}[theorem]{Conjecture}
\newtheorem{remark}[theorem]{Remark}

\begin{document}
\noindent {\bf    , Issue  ()}
\section*{Curious Properties of Iterative Sequences}
\vspace{-.30cm}

\begin{center}
\textbf{Shoei Takahashi, Unchone Lee, Hikaru Manabe, Aoi Murakami}\footnote{%
Students at Keimei Gakuin High School in Japan}
\textbf{Daisuke Minematsu, Kou Omori, }\footnote{%
Daisuke Minematsu and Kou Omori are  independent researchers in Japan.}
\textbf{ and Ryohei Miyadera}\footnote{%
Ryohei Miyadera, Ph.D., is a mathematician and mathematics advisor at Keimei Gakuin in Japan.}
\end{center}

\section{Introduction}
In this study, several interesting iterative sequences were investigated.

First, we define the iterative sequences.
Let $\mathbb{N}$ be a set of natural numbers. We fix function $f(n) \in \mathbb{N}$ for $n \in \mathbb{N}$. 
An iterative sequence starts with $n \in \mathbb{N}$, and calculates the sequence
$f(n), f^2(n) = f(f(n)), \cdots, f^{m+1}(n) = f(f^m(n)) \cdots$. We then search for interesting features in this sequence.

Iterative sequences can be useful topics for professional or amateur mathematicians if they choose the proper problem. The authors present new findings in this study.

We begin with the Collatz conjecture, which is one of the most well-known unsolved problems concerning iterative sequences. A mathematician Lothar Collatz presented the following prediction $86$ years ago:
\begin{prediction}
For any natural number $n$, if we apply the function in (\ref{collatz}) repeatedly to $n$, we eventually reach number $1$.   
\begin{equation}
  f(x)=
  \begin{cases}
    \frac{x}{2} & \text{if $x$ is even,} \\
    3x+1      & \text{if $x$ is odd.}\label{collatz}
  \end{cases}
\end{equation}
\end{prediction}

\begin{example}\label{examcollatz}
This function is applied to several numbers. \\
$(i)$ If we start with $67$, then from (\ref{collatz}), we have
$67, 202, 101, 304, 152, 76, 38, 19, 58, 29, 88, 44, 22,$\\
$11, 34, 17, 52, 26, 13, 40, 20, 10, 5, 16, 8, 4, 2, 1$.\\
$(ii)$ If we start with $36$, then from (\ref{collatz}), we have
$36, 18, 9, 28, 14, 7, 22, 11, 34, 17, 52, 26, 13, 40, 20,$
$10, 5, 16, 8, 4, 2, 1$.\\

\begin{minipage}[t]{0.42\textwidth}
\begin{center}
\begin{figure}[H]
\includegraphics[height=3cm]{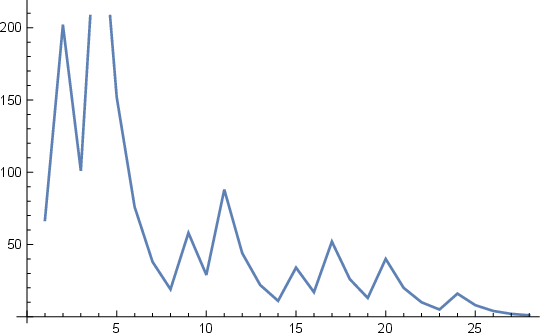}
\caption{Graph of the sequence in (i) of Example \ref{examcollatz}. }\label{chocoof5and3}
\end{figure}
\end{center}
\end{minipage}
\hfill
\begin{minipage}[t]{0.42\textwidth}
\begin{center}
\begin{figure}[H]
\includegraphics[height=3cm]{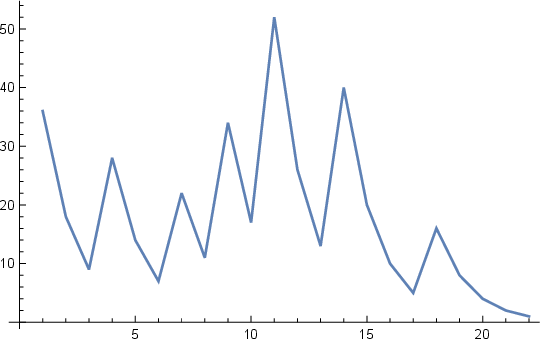}
\caption{Graph of the sequence in (ii) of Example \ref{examcollatz}. }\label{nimof5and3}
\end{figure}
\end{center}
\end{minipage}

The calculations of $(i)$ and $(ii)$ were carried out using the following program:
\url{https://github.com/Shoei256/Curious-Properties-of-Iterative-Sequences-}\\
\url{programs/blob/main/others/Collatz%20problem.py}

\end{example}

Although several mathematicians have studied this problem for 86 years, there is still no proof or counterexample for this prediction.

Because Collatz conjecture is simple and well-known, it has attracted the attention of many people, and many 
amateur mathematicians spend a considerable amount of time on it.

It is good to appreciate this conjecture; however, it is not wise for most people to attempt to resolve it.
The eminent mathematician Paul Erdős said about the Collatz conjecture: "Mathematics may not be ready for such problems." 
For the detail of Collatz conjecture, see \cite{collatzmath}.
Collatz conjecture was independently proposed by other mathematicians Stanisław Ulam, Shizuo Kakutani, Bryan Thwaites, and Helmut Hasse.

In the remainder of this article, we study iterative sequences, such as Kaprekar's routine, the digit factorial process, and the digit power process. 
The authors presented new variants of Kaprekar's routine.

Regarding the digit power and digit factorial processes, two of the authors of the present article have already presented some results in
\cite{IMS2004} and \cite{vismath04}, and it this article we study these processes from a new perspective.
For further information on iterative sequences, see \cite{mathama}.

\section{Kaprekar's routine}
Kaprekar \cite{Kaprekar} discovered an interesting property of an iterative sequence, which was subsequently named Kaprekar's routine.

In this iterative sequence, we start with a four-digit number whose four digits are not the same.
If we repeatedly subtract the highest and lowest numbers constructed from a set of four digits, 
we eventually obtain $6174$. We call $6174$ as the Kaprekar's constant. 

\begin{example}
 $(i)$ If we start with 1234, we obtain:
$4321-1234=3987$, $8730-0378=8352$, and $8532-2358=6174$.
Subsequently, because $7641-1467=6174$, we obtain the same number, $6174$. \\   
$(ii)$ If we start with 1001, we obtain:
$1100-11=1089$, $9810-189=9621$, $9621-1269=8352$, and 
$8532-2358=6174$. Subsequently, we obtained the same number, $6174$.
\end{example}

For details of Kaprekar's routine, see \cite{Kaprekarmath}. There seem to be many things that can be discovered in Kaprekar's routine. See Section \ref{kaprekarva}.

\begin{example}
The authors present a Python program for Kaprekar's routine.
This program searches for four-digit numbers that do not converge to Kaprekar's constant.

The program is available at \url{https://github.com/Shoei256/Curious-}\\
\url{Properties-of-Iterative-Sequences-programs/blob/main/Kaprekar/}\\
\url{Kaprekar.py}.\\
The output is as follows:
\small{
\begin{verbatim}
Four-digit numbers that do not converge to Kaprekar's constant (6174):
[1111, 2222, 3333, 4444, 5555, 6666, 7777, 8888, 9999]
\end{verbatim}}
\end{example}

\section{Variants of Kaprekar's routine}\label{kaprekarva}
Here, we studied a few variants of Kaprekar's routine.
\subsection{The case of base 10}
\begin{definition}
Let $n,u,v$ be  natural numbers. We 
sort the digits of 
$n$ in descending order, and let $\alpha_u$ be the $u$-th largest number.
Subsequently, we sort the digits of 
$n$ in ascending order and let $\beta_v$ be the $v$-th smallest number.
Then, we subtract  $\alpha_u -\beta_v$  to produce the next number, $K_{u,v}(n)$.
\end{definition}

\begin{example}
$(i)$ For any $3$-digit number $n$, $\{K_{2,2}^t(n):t = 1, 2, \dots ,\}$ reaches a fixed point $450$. \\
$(ii.1)$ For any $4$-digit number $n$, $\{K_{3,1}^t(n):t = 1, 2, \dots ,\}$ reaches a fixed point $4995$. \\
$(ii.2)$ For any $4$-digit number $n$, $\{K_{1,2}^t(n):t = 1, 2, \dots ,\}$ reaches one of the fixed points: $9045,4995,4997$. \\
$(iii.1)$ For any $5$-digit number $n$, $\{K_{3,2}^t(n):t = 1, 2, \dots ,\}$ reaches a fixed point $49995$. \\
$(iii.2)$ For any $5$-digit number $n$, $\{K_{4,1}^t(n):t = 1, 2, \dots ,\}$ reaches a fixed point $62748$. \\
The following computer program performs the calculations for these iterative sequences.\\
\url{https://github.com/Shoei256/Curious-Properties-of-Iterative-}\\
\url{Sequences-programs/blob/main/Kaprekar/base10.py}
\end{example}

\subsection{The case of base 2}
In this subsection, we treat all the numbers in base $2$.
\begin{definition}
We choose any natural number $n$ in base $2$. We sort the digits of 
$n$ in descending order, and let $\alpha_2$ be the second-largest number.
Subsequently, we sort the digits of 
$n$ in ascending order and let $\beta_2$ be the second-smallest number.
These numbers may have leading zeros, which are  retained. 
Then, we subtract  $\alpha_2 -\beta_2$  to generate the next number, $K_{2,2}(n)$.
\end{definition}

\begin{conjecture}\label{conjkapre}
Let $n$ be a natural number, such that $m \geq 3$. \\
$(i)$ When the digit length is $2m$, the sequence $\{K_{2,2}^t(n):t = 1, 2, \dots ,\}$ reaches a fixed point $2^{2m}-3 \times 2^m+1$. \\
$(ii)$ When the digit length is $2m+1$, the sequence $\{K_{2,2}^t(n):t = 1, 2, \dots ,\}$ enters a loop of two numbers such that $\{2^{2m+1}-7 \times 2^m+2^{m-2}+1, 2^{2m+1}-7 \times 2^m+10 \times 2^{m-2}+1\} $.
\end{conjecture}

For the numbers with digit lengths of $6,7,8,9$ in base $2$, Conjecture \ref{conjkapre} is valid as presented in Table \ref{Table2a}.

The following computer program performs the calculations and the results are listed in Table \ref{Table2a}.
\url{https://github.com/Shoei256/Curious-Properties-of-Iterative-}\\
\url{Sequences-programs/blob/main/Kaprekar/base2.py}

\begin{table}[H]
\caption{}
\label{Table2a}
\vspace{0.1cm}
\begin{tabular}{|c|c|c|c|}
\hline
Base & Digit & fixed points & loops  \\ 
  & length &  &   \\ \hline
  & 4 & 101, 10 & none \\ \cline{2-4}
2 & 5 &  none  & \{110,1111\} \\ \cline{2-4}
  & 6 & $101001=2^{2 \cdot 3}-3 \cdot 2^3+1$ &  none    \\ \cline{2-4}
  & 7 & none & $\{1001011, 1011101\}$ \\ 
  &  &  & $=\{2^{2\cdot 3+1}-7 \cdot 2^3+10 \cdot 2^{3-2}+1 $ \\ 
  &  &  & $, 2^{2 \cdot 3+1}-7 \cdot 2^3+ 2^{3-2}+1  \}$ \\ \cline{2-4}
  & 8 & $11010001=2^{2 \cdot 4}-3 \cdot 2^4+1$ & none  \\ \cline{2-4}
  & 9 & none & \{110010101,110111001\}   \\ 
  &  &  & $\{2^{2\cdot 4+1}-7 \cdot 2^4+10 \cdot 2^{4-2}+1  $  \\ 
  &  & & $, 2^{2\cdot 4+1}-7 \cdot 2^4+ 2^{4-2}+1 \}$ \\ \hline
\end{tabular}
\end{table}

\section{Digits Factorial Process}
This process was studied by \cite{vismath04} and \cite{mathama}; however, in this study, we present new results.
Lemma \ref{looplemma1}, Theorem \ref{theoremfordfp}, and Lemma \ref{smallerthan1010} are new results.

 We define  $\mathbf{dfp}(n) = \sum^{m}_{k=1} n_k!$
, where $ \left\{n_1,n_2,\cdots ,n_m \right\}$ is the list of the digits of an integer n. We start with any non-negative integer n, and repeatedly apply the function $\mathbf{dfp}$ to generate a sequence of integers.

\begin{example}
$\mathbf{dfp}(123)=1!+2!+3!=9$,
$\mathbf{dfp}^2(123)= \mathbf{dfp}(9)=9!=362880$,
$\mathbf{dfp}^3(123)$\\
$= \mathbf{dfp}(\mathbf{dfp}^2(123))=\mathbf{dfp}(362880)=3!+6!+2!+8!+8!+0!=81369, \cdots .$   
\end{example}

\begin{lemma}\label{loopsofdfp}
The following fixA, fixB, fixC, and fixD are fixed points, and
loop2A, loop2B, and loop3 are the loops of
$\mathbf{dfp}$.\\
$(i)$ fixA=$\{1\}$, fixB=$\{2\}$, fixC=$\{145\}$, fixD=$\{40585\}$\\
$(ii)$  loop2A=$\{871, 45361\}$, loop2B=$\{872, 45362\}$\\
$(iii)$  loop3=$\{169, 363601, 1454\}$
\end{lemma}
\begin{proof}
$(i)$ $\mathbf{dfp}(1)=1!=1$, $\mathbf{dfp}(2)=2!=2$, $\mathbf{dfp}(145)=1!+4!+5!=145$, $\mathbf{dfp}(40585)$\\
$=4!+0!+5!+8!+5!=40585$.\\
$(ii)$ 
$\mathbf{dfp}(871)=8!+7!+1!=45361, \mathbf{dfp}(45361)=4!+5!+3!+6!+1!=871$, 
$\mathbf{dfp}(872)$\\
$=8!+7!+2!=45362, \mathbf{dfp}(45362)=4!+5!+3!+6!+2!=872$.\\
$(iii)$ $\mathbf{dfp}(169)=1!+6!+9!=363601, \mathbf{dfp}(363601)=3!+6!+3!+6!+0!+1!=1454$\\
$, \mathbf{dfp}(1454)=1!+4!+5!+4!=169$. 
\end{proof}

\begin{remark}
Loops loop2A, loop2B, and loop3 are
registered as 
\cite{dfp2a}, \cite{dfp2b}, \cite{dfp3} in the On-Line Encyclopedia of Integer Sequences.
\end{remark}

In \cite{mathama}, Lehmann wrote that the abovementioned loops are the only loops in natural numbers smaller than $2000000$; however, we can prove that these are the only loops for all natural numbers. 

\begin{lemma}\label{lemmafordfp0}
For a natural number $m$ such that $m \geq 2$ we have
\begin{equation}
\frac{10^{m}-1}{m\times 10!} <  \frac{10^{m+1}-1}{(m+1)\times 10!}.\label{inequality10}
\end{equation}  
\end{lemma}
\begin{proof}
As for $m \geq 2$, we have
\begin{equation}
\frac{m+1}{m} < 2 <  10 < \frac{10^{m+1}-1}{10^m-1}, \nonumber
\end{equation} 
we have (\ref{inequality10}).
\end{proof}

\begin{lemma}\label{lemmafordfp1}
For a natural number $m$ such that $m \geq 8$, we have
\begin{equation}
\frac{10^{m}-1}{10} > m \times 9!.\label{inequality1}
\end{equation}  
\end{lemma}
\begin{proof}
The sequence 
$\frac{10^{m}-1}{m\times 10!}= 0.0000136409, 0.0000917659, 0.000688864, 0.00551141, $\\
$0.0459288, 0.393676, 3.44466$ for $m$ $=2,3, \cdots, 8$.
From Lemma \ref{lemmafordfp0}, the sequence
$\{ \frac{10^{m}-1}{m\times 10!}:m = 2,3,4, \cdots \}$ is increasing. Hence, we have
\begin{equation}
\frac{10^m-1}{m \times 10!} > 1 \nonumber
\end{equation}
for $m \geq 8$.
Therefore, we obtain (\ref{inequality1}).
\end{proof}

\begin{lemma}\label{dfpsmalln}
For a natural number $n$ such that $n \geq 10^7$, we have
\begin{equation}
\mathbf{dfp}(n) < n.\label{dfpsmallnn}
\end{equation}  
\end{lemma}
\begin{proof}
For $n \geq 10^7$, there exists a natural number $m$ such that
\begin{equation}
m \geq 7 \nonumber
\end{equation}
and
\begin{equation}
10^m \leq n < 10^{m+1}-1.\label{msmalln10m1}
\end{equation}
Subsequently, from (\ref{msmalln10m1}), Lemma \ref{lemmafordfp1} and the definiton of $\mathbf{dfp}$, 
\begin{equation}
\mathbf{dfp}(n) \leq (m+1)9! < \frac{10^{m+1}-1}{10} < 10^m \leq n.\nonumber
\end{equation}
Therefore, we have (\ref{dfpsmallnn}).
\end{proof}

\begin{lemma}\label{looplemma1}
If we start with an arbitrary natural number $n$ and repeatedly apply $\mathbf{dfp}$, we enter a set of numbers smaller than $10^{7}$. Subsequently, we do not leave the set. In particular, we fall into a loop or reach a fixed point.
\end{lemma}
\begin{proof}
By Lemma \ref{dfpsmalln}, we have $\mathbf{dfp}(n) < n$ for $n \geq 10^7$.
If $n < 10^7$, $\mathbf{dfp}(n) \leq 7 \times 9! $\\
$ < 2.54 \times 10^6 < 10^7$.
Therefore, starting with an arbitrary natural number $n$ such that
$n \geq 10^7$, we obtain a strictly decreasing sequence
$\{\mathbf{dfp}^t(n):t=1,2, \cdots \}$, which enters a set of numbers smaller than $10^7$.
Once the sequence enters a set of numbers smaller than $10^7$, it does not leave this set.
Because this set contains only $10^7-1$ numbers, there exists $s,t \in \mathbb{N}$ such that
$\mathbf{dfp}^s(n)=\mathbf{dfp}^t(n)$.
Therefore, either we enter a loop or reach a fixed point.
\end{proof}

From Lemma \ref{looplemma1}, for any natural number $n$, the sequence $\{\mathbf{dfp}^t(n):t=1,2, \cdots \} $ enters a loop or reaches a fixed point; however, this lemma does not provide any information regarding the properties of these loops or fixed points.

In fact, we have the following theorem:

\begin{theorem}\label{theoremfordfp}
Suppose that we start with any natural number and repeatedly apply the function $\mathbf{dfp}$. Eventually, we fall into a loop or reach a fixed point in Lemma \ref{loopsofdfp}. 
\end{theorem}
We prove this theorem in Section \ref{computerc}.

\section{Digits Power Process}
This process was first studied by present authors in \cite{IMS2004}.

 We define $\mathbf{\mathbf{dpp}}(n) = \sum^{m}_{k=1} n_k^{n_k}$
, where $ \left\{n_1,n_2,\cdots, n_m \right\}$ denotes the list of digits of an integer $n$. We start with any non-negative integer $n$ and repeatedly apply the function $\mathbf{\mathbf{dpp}}$ to generate a sequence of integers. 

\begin{example}
If we start with $123$ and apply $\mathbf{dpp}$ repeatedly, we obtain
$\mathbf{\mathbf{dpp}}(123)$\\
$=1^1+2^2+3^3=32$,
$\mathbf{\mathbf{dpp}}^2(123)= \mathbf{\mathbf{dpp}}(32)=3^3+2^2=31$,
$\mathbf{\mathbf{dpp}}^3(123)= \mathbf{\mathbf{dpp}}(\mathbf{\mathbf{dpp}}^2(123))$\\
$ =\mathbf{\mathbf{dpp}}(31) =3^3+1^1=28, \cdots$.  
\end{example}

\begin{lemma}\label{lemmadpp10}
For a natural number $m$ such that $m \geq 2$, we have
\begin{equation}
\frac{10^{m}-1}{m \times 9^9} < \frac{10^{m+1}-1}{(m+1) \times 9^9}.\label{inequality10b}
\end{equation}
\end{lemma}
\begin{proof}
As for $m \geq 2$, we have
\begin{equation}
\frac{m+1}{m} < 2 <  10 < \frac{10^{m+1}-1}{10^m-1},\nonumber
\end{equation} 
we obtain (\ref{inequality10b}).
\end{proof}

\begin{lemma}
For a natural number $m$ such that $m \geq 11$, we have
\begin{equation}
\frac{10^{m}-1}{10} > m \times 9^9.\label{10m99}
\end{equation}  
\end{lemma}
\begin{proof}
$\frac{10^{m}-1}{m \times 9^9}= 0.0000430195, 0.000368739, 0.00322647, 0.0286797, 0.258117, 2.34652$ for $m=6,7, \cdots, 11$.
From Lemma \ref{lemmadpp10}, the sequence
$\{\frac{10^{m}-1}{m \times 9^9}:m = 2,3,4, \cdots \}$ is increasing. Hence,
we have
$\frac{10^{m}-1}{m \times 9^9}>1$ for  any natural number $m \geq 11$.
Therefore, we have (\ref{10m99}).
\end{proof}

\begin{lemma}\label{lemmadppsmn}
For a natural number $n$ such that $n \geq 10^{10}$, we have
\begin{equation}
\mathbf{\mathbf{dpp}}(n) < n.
\end{equation}  
\end{lemma}
\begin{proof}
Let $n \geq 10^{10}$. Then, there exists a natural number $m$ such that
$m \geq 10$
and
\begin{equation}
10^m \leq n < 10^{m+1}-1.
\end{equation}
Therefore, 
\begin{equation}
\mathbf{dpp}(n) \leq (m+1)9^9 < \frac{10^{m+1}-1}{10} < 10^m \leq n.
\end{equation}
\end{proof}

\begin{definition}\label{dpploop}
fixA=\{1\};fixB=\{3435\};
loop2=\{421845123,16780890\};\\
loop3=\{16777500,2520413,3418\};\\
loop8=\{809265896,808491852,437755524,1657004,873583,
34381154,16780909,792488396\};\\
loop11=\{791621579,776537851,19300779,776488094,422669176,
388384265,50381743,\\
17604196,388337603,34424740,824599\};\\
loop40=\{793312220,388244100,33554978,405027808,34381363,
16824237,17647707,\\
3341086,16824184,33601606,140025,3388,
33554486,16830688,50424989,791621836,\\
405114593,387427281,
35201810,16780376,18517643,17650825,17653671,1743552,
830081,\\
33554462,53476,873607,18470986,421845378,34381644,
16824695,404294403,387421546,\\
17651084,17650799,776537847,20121452,3396,387467199\};\\
loop97=\{1583236420,16827317,18470991,792441996,1163132183,
16823961,404291050,\\
387424134,17601586,17697199,1163955211,
387473430,18424896,421022094,387421016,\\
17647705,2520668,
16873662,17740759,389894501,808398820,454529386,404251154,\\
7025,826673,17694102,388290951,808398568,454579162,
388297455,421805001,16780606,\\
17740730,2470915,388247419,
421799008,792442000,388244555,33564350,53244,3668,\\
16870555,17656792,389164017,405068190,404247746,1694771,389114489,808395951,\\
808401689,437799052,776491477,390761830,
405067961,388340728,51155506,59159,\\
774847229,406668854,33698038,421021659,387470537,19251281,404200841,16777992,\\
777358268,36074873,18471269,405068166,16920568,404294148,
404198735,405024914,\\
387424389,421799034,775665066,1839961,
791664879,793358849,809222388,437752177,\\
3297585,405027529,388250548,50338186,33604269,387514116,17650826,17697202,\\
389114241,404198251,404201349,387421291,405021541,6770,
1693743,388290999\};
\end{definition}

\begin{remark}
loop2, loop3, loop8, loop11 and loop40 are registered by one of the authors of the present paper in the On-Line Encyclopedia of Integer Sequences.
See \cite{integersequence2}, \cite{integersequence3}, \cite{integersequence8}, \cite{integersequence11}, and \cite{integersequence40}.
\end{remark}

\begin{lemma}\label{smallerthan1010}
If we start with an arbitrary natural number $n$ and apply $\mathbf{dpp}$ repeatedly, we will enter a set of numbers smaller than $10^{10}$. Then, we do not leave the set. In particular, we fall into a loop or reach a fixed point.
\end{lemma}
\begin{proof}
From Lemma \ref{lemmadppsmn}, $\mathbf{dpp}(n) < n$ for $n \geq 10^{10}$.
If $n < 10^{10}$, $\mathbf{dpp}(n) \leq 10 \times 9^9 < 3.8743 \times 10^9 < 10^{10}$.
Therefore, if we start with an arbitrary natural number $n$ such that $n \geq 10^{10}$, we obtain a strictly decreasing sequence
$\{\mathbf{dpp}^t(n):t=1,2, \cdots \}$, which enters a set of numbers less than $10^{10}$.
Once a sequence enters this set, it never exits. 
Because this set contains only $10^{10}-1$ numbers, there exist $s,t \in \mathbb{N}$ such that
$\mathbf{dpp}^s(n)=\mathbf{dpp}^t(n)$.
Therefore, we either enter a loop or reach a fixed point.
\end{proof}

\begin{theorem}\label{theoremfordpp}
Suppose that we start with any natural number and repeatedly apply the function $dpp$. Eventually, we fall into a loops or reach a fixed point in Definition \ref{dpploop}.  
\end{theorem}
We prove this lemma in Section \ref{computerc}.

\section{Proofs by Computers}\label{computerc}
We prove Theorem \ref{theoremfordfp} and Theorem \ref{theoremfordpp} by computational analysis.
From Lemma \ref{looplemma1}, we need to prove Theorem \ref{theoremfordfp} for any natural number $n < 10^7$, and  from Lemma \ref{smallerthan1010}, we need to prove Theorem \ref{theoremfordpp} for 
for any natural number $n < 10^{10}$. However, we study a method to simplify the calculation.

\begin{lemma}\label{lemmaforcom}
Let $r$ be a natural number, and we select $r$ that is not necessarily distinct from the set $\{1,2,3,4,5,6,7,8,9\}$. Then, the total number of combinations is 
\begin{equation}
\binom{r+8}{r}.\label{r8r}
\end{equation}
\end{lemma}
\begin{proof}
The number of ways for choosing $r$
 elements from a set of $n$
 elements if repetitions are allowed is
\begin{equation}
\binom{n+r-1}{r}.\nonumber
\end{equation}
Therefore, we have (\ref{r8r}).
\end{proof}

\begin{lemma}\label{smallersets}
$(i)$ To prove Theorem \ref{theoremfordfp}, we only need to consider $11439$ numbers. \\
$(ii)$ To prove Theorem \ref{theoremfordpp}, we must consider $92377$ numbers.
\end{lemma}
\begin{proof}
$(i)$ From Lemma \ref{looplemma1}, we only need to consider numbers smaller than $10^7$, and 
$\{n:n < 10^7\} = \{n: \text{the length of digits is less than } 8\}.$
In the digit factorial process, the order of the digits is not relevant to the creation of the next number. Because $0!=1!=1$, we treat  $0$ and $1$ as the same number.
Therefore, from Lemma \ref{lemmaforcom}, 
the total number of numbers that must be considered is
\begin{equation}
\sum^{7}_{r=1}\binom{r+8}{r}=11439.\nonumber
\end{equation}
$(ii)$ From Lemma \ref{smallerthan1010}, we only need to consider numbers smaller than $10^{10}$, and
$\{n:n < 10^{10}\} = \{n: \text{the length of digits is less than } 11\}.$
In the digit power process, the order of the digits is not relevant to the creation of the next number. Because $0^0=1^1=1$, we can treat  $0$ and $1$ as the same number.
Therefore, from Lemma \ref{lemmaforcom}, the total number of numbers that must be considered is
\begin{equation}
\sum^{10}_{r=1}\binom{r+8}{r}=92377.\nonumber
\end{equation}
\end{proof}

\begin{computer}
We prove Theorem \ref{theoremfordfp} and Theorem \ref{theoremfordpp}
using the following program. This program is based on Lemma \ref{smallersets}.\\
\url{https://github.com/Shoei256/Curious-Properties-of-Iterative}\\
\url{-Sequences-programs/tree/main/others}
\end{computer}

\section{Prospect for Future Research}
Iterative functions are a good topic for high school or undergraduate research; however, studying well-known unsolved problems such as the Collatz Problem is not a good idea.

A good way to conduct research is to change some parts of the original problem, such as that in Section \ref{kaprekarva}.

We can also use $xor$ instead of $+$ in digit factorial or digit power processes.
After you study this process, please see the following program.\\
\url{https://github.com/Shoei256/Curious-Properties-of-Iterative-}\\
\url{Sequences-programs/tree/main/others/Digits%20Factorial%20XOR.py}

\end{document}